\documentclass[reqno,11pt]{amsart}
\usepackage{amsmath,amssymb,amsthm, verbatim}
\usepackage{url}
\usepackage[foot]{amsaddr}
\usepackage[usenames, dvipsnames]{color}

\usepackage[letterpaper,hmargin=1in,vmargin=1in]{geometry}

\usepackage{graphicx}
\usepackage{enumerate,pinlabel}
\usepackage{mathrsfs,graphicx,url}
\usepackage[usenames,dvipsnames]{xcolor}
\usepackage[colorlinks=true,linkcolor=Black,citecolor=Black, urlcolor=Black,pagebackref]{hyperref}

\numberwithin{equation}{section}

\theoremstyle{plain}
\newtheorem{theorem}{Theorem}[section]

\newtheorem*{theorem*}{Theorem}
\newtheorem*{lemma*}{Lemma}
\newtheorem{lemma}[theorem]{Lemma}
\newtheorem{proposition}[theorem]{Proposition}
\newtheorem{corollary}[theorem]{Corollary}

\theoremstyle{definition}

\theoremstyle{remark}

\newcommand{\ep}{\varepsilon}
\newcommand{\NN}{\mathbb{N}}
\newcommand{\R}{\mathbb{R}}

\newcommand{\C}{\mathbb{C}}
\newcommand{\D}{\mathcal{D}}

\newcommand\supp{\mathop{\rm supp}}

\newcommand\real{\mathop{\rm Re}}
\newcommand\imag{\mathop{\rm Im}}

\newcommand*{\defeq}{\mathrel{\vcenter{\baselineskip0.5ex \lineskiplimit0pt

                     \hbox{\scriptsize.}\hbox{\scriptsize.}}}%
                     =}
                     
\newcommand*{\qefed}{=\mathrel{\vcenter{\baselineskip0.5ex \lineskiplimit0pt

                     \hbox{\scriptsize.}\hbox{\scriptsize.}}}}

\begin{document}
\title[Exponential decay for BV coefficients]{Exponential time-decay for a one dimensional wave equation with coefficients of bounded variation}

\author{Kiril Datchev}
\address{Department of Mathematics, Purdue University,  West Lafayette, IN, 47907-2067, USA}
\email{kdatchev@purdue.edu}
\author{Jacob Shapiro}
\address{Department of Mathematics, University of Dayton, Dayton, OH 45469-2316, USA}
\email{jshapiro1@udayton.edu}
\thanks{Second author is the corresponding author}


\maketitle

\begin{abstract}
We consider the initial-value problem for a one-dimensional wave equation with coefficients that are positive, constant outside of an interval, and have bounded variation (BV). Under the assumption of compact support of the initial data, we prove that the local energy decays exponentially fast in time, and provide the explicit constant to which the solution converges. The key ingredient of the proof is a high frequency resolvent estimate for an associated Helmholtz operator with a BV potential.
 \end{abstract}

\section{Introduction and statement of results} \label{introduction}

This paper establishes exponential local energy decay for the solution of the following one dimensional wave equation, with compactly supported initial data:
\begin{equation} \label{wave eqn}
\begin{cases}
\beta(x) \partial_t^2 w(x,t) - \partial_x (\alpha(x) \partial_x w(x,t)) = 0, \qquad (x,t) \in \R \times (0, \infty), \\
w(x,0) = w_0(x),  \\
\partial_t w(x,0) = w_1(x), \\
\supp w_0, \, \supp w_1 \subseteq (-R,R), \qquad R > 0. 
\end{cases} 
\end{equation}
Here, the coefficients $\alpha, \beta : \R \to (0,\infty)$ have bounded variation (BV). We suppose also
\begin{equation} \label{infs positive}
    \inf_\R \alpha, \, \inf_\R \beta > 0,
\end{equation}
and that there exist $R_0, \, \alpha_0, \, \beta_0 > 0$, so that
\begin{equation} \label{perturbations}
    \alpha(x) = \alpha_0, \, \beta(x) = \beta_0, \qquad |x| \ge R_0.
 \end{equation} 
 
To begin, we address the well-posedness of \eqref{wave eqn} via the spectral theorem for self-adjoint operators. Let $\mathcal H$ be the Hilbert space  $L^2(\mathbb R; \beta(x)dx)$ equipped with the inner product 
\[\langle u,v \rangle_{\mathcal H} \defeq \int_{\R}\overline{u}(x) v(x) \beta(x) dx.\]
(Note that $L^2(\mathbb R; \beta(x)dx) = L^2(\mathbb R; dx)$ as sets, and their respective norms generate the same topology, since $\beta$ has positive upper and lower bounds.) Define the symmetric, nonnegative differential operator
\begin{equation} \label{H}
    H u \defeq - \beta^{-1} \partial_x (\alpha  \partial_x u),
\end{equation}
 with domain  $\mathcal D(H) \defeq \{u \in L^2(\R) : u, \partial_x u \in L^2(\R) \cap L^\infty(\R), \text{ and } \partial_x(\alpha \partial_xu) \in L^2(\R) \}$. We will see from Lemma \ref{self adjointness lemma} in Section \ref{wtd resolv est section} that $H$ is self-adjoint with respect to $\mathcal{D}(H)$. It is also conveniently the case that $D(H^{1/2})$ coincides with the Sobolev space $H^1(\R)$ \cite{re22b}. For completeness, we prove this fact in Appendix \ref{H1 appendix}. 
 
 Thus, for initial conditions $w_0 \in \mathcal{D}(H)$, $w_1 \in \mathcal{D}(H^{1/2})$, 
 \begin{equation} \label{soln spectral thm}
w(t) = w(\cdot, t) = \cos(t H^{1/2}) w_0 + \frac{\sin(tH^{1/2})}{H^{1/2}} w_1.
\end{equation}
 is the unique function $w \in C^2((0,\infty), \mathcal{H})$ with $w(0) = w_0$, $\partial_tw(0) = w_1$, and for all $t > 0$, $w(t) \in \mathcal{D}(H)$ and $\partial^2_t w(t) +Hw(t)=0$. 
 
 \begin{theorem} \label{LED thm}
 Let $\alpha, \, \beta : \R \to (0,\infty)$ have bounded variation and satisfy \eqref{infs positive} and \eqref{perturbations}. Suppose $w_0 \in \mathcal{D}(H), \, w_1 \in \mathcal{D}(H^{1/2})$, and $\supp w_0,\, \supp w_1 \subseteq (-R, R)$ for some $R> 0$. Let $w(t)$ be given by \eqref{soln spectral thm}. For any   $R_1 > 0$, there exist $C, c > 0$ so that 
\begin{equation} \label{LED}
\begin{gathered}
 \| w(\cdot, t) - w_\infty \|_{H^1(-R_1,R_1)} +
    \|\partial_t w(\cdot, t)  \|_{L^2(-R_1,R_1)} \\
   \le C e^{-c t} (\| w_0\|_{H^1(\R)} + \|w_1\|_{L^2(\R)}), \qquad t > 0,
   \end{gathered}
\end{equation}
where 
   \begin{equation} \label{w infty}
       w_\infty \defeq  \frac{1}{2(\alpha_0 \beta_0)^{1/2}}\int_{\R} w_1(x) \beta(x) dx.
   \end{equation}
\end{theorem}
 
 Theorem \ref{LED thm} is motivated by the recent article \cite{agpp22}. There, the authors prove \eqref{LED}, with an explicit constant $c$ depending on $\alpha$ and $\beta$, provided that $\alpha$ and $\beta$ are Lipschitz continuous, bounded from above and below by positive constants, and satisfy \eqref{perturbations}. Our result includes natural examples such as cases where $\alpha$ and $\beta$ are piecewise constant and it is easy to see that the exponential decay rate in \eqref{LED} cannot in general be improved to any superexponential rate. See \cite{biz16} for dispersive and Strichartz estimates for one dimensional wave equations with BV coefficients.
 
 To prove Theorem \ref{LED thm}, it suffices to show \eqref{LED} and \eqref{w infty} in the special case
  \begin{equation} \label{perturbations of identity} 
        \alpha(x) = \beta(x) = 1, \qquad |x| \ge R_0.
  \end{equation}
  Indeed, if $w(x,t)$ solves \eqref{wave eqn} for initial conditions $w_0, \, w_1$ and general $\alpha$ and $\beta$, then the function  $u(x,t) \defeq w(\sqrt{\alpha_0}) x, \sqrt{\beta_0}t)$ solves $ (\beta(\sqrt{\alpha_0}x)/\beta_0)\partial_t^2 u - \partial_x ((\alpha(\sqrt{\alpha_0}x)/\alpha_0) \partial_x u) = 0$ with initial conditions $u(x,0) = w_0(\sqrt{\alpha_0}x), \, \partial_t u(x,0) = \sqrt{\beta}_0 w_1(\sqrt{\alpha_0}x)$. Then \eqref{perturbations of identity} applies, giving that $u$ decays according to \eqref{LED} and \eqref{w infty}. The asserted decay for $w$ follows by a change of variables.
  
  For the wave equation with \textit{constant} coefficients and compactly supported initial conditions, it follows readily from D'Alembert's formula that solution to \eqref{wave eqn} converges to $w_\infty$ in finite time. However, for variable coefficients, exponential decay is a typical scenario. This occurs in the setting of reflection and transmission, e.g., when $\alpha \equiv 1$ and $\beta$ assumes precisely two values.
  
 In dimensions two and higher, the recent works \cite{chik20, sh18} treat local energy decay for wave equations with Lipschitz coefficients. Though in higher dimensions, logarithmic, rather than exponential decay, is optimal in general. The study of energy decay more broadly has a long history, going back to the foundational work of Morawetz, Lax--Phillips, and Vainberg \cite{mor, lmp, lp, vai}, which we will not attempt to review here. The reader may consult \cite{bu98, hizw17,sh18,dz} for more historical background and references.

 We prove Theorem \ref{LED thm} by analyzing $H$ as a \textit{black box Hamiltonian} in the sense of Sj\"ostrand and Zworski \cite{sz91}. In particular, \eqref{perturbations of identity} implies that for any $\chi \in C_0^\infty(\mathbb R; [0,1])$ that is identically one near $[-R_0, R_0]$, the cutoff resolvent
\begin{equation} \label{continued resolv}
    \chi R(\lambda) \chi \defeq \chi (H- \lambda^2)^{-1} \chi : \mathcal{H} \to \mathcal{D}(H)
\end{equation}
continues meromorphically from $\imag \lambda > 0$ to the complex plane. (Here, we equip $\mathcal{D}(H)$ with the graph norm $u \mapsto (\|u\|^2_{\mathcal{H}} + \|Hu\|^2_{\mathcal{H}})^{1/2}$.) In particular, we establish the following high frequency bound.

\begin{theorem} \label{unif resolv est thm}
 Suppose $\alpha, \beta : \R \to (0, \infty)$ have bounded variation and obey \eqref{infs positive} and \eqref{perturbations of identity}. For any $\chi \in C_0^\infty(\mathbb R; [0,1])$ that is identically one near $[-R_0, R_0]$, there exists $C, \, \lambda_0, \, \ep_0 > 0$ so that
 \begin{equation} \label{unif resolv est}
      \|\chi R(\lambda) \chi\|_{\mathcal{H} \to \mathcal{H}} \le C|\real \lambda |^{-1}, 
 \end{equation}
 whenever $|\real \lambda| \ge \lambda_0$, and $|\imag \lambda| \le \ep_0$.
\end{theorem}

In Section \ref{high energy bound section}, we achieve \eqref{unif resolv est} by rescaling $H - \lambda^2$ semiclassically, see \eqref{A}, and apply a resolvent estimate for a Schr\"odinger operator with a BV potential, namely Theorem \ref{nontrap thm oned} in Section \ref{wtd resolv est section}. The proof of Theorem \ref{nontrap thm oned} uses a positive commutator argument that relies on some basic calculus facts for BV functions. We collect these facts in Section \ref{bv review section}, and prove them in Appendix \ref{BV appendix}. Finally, in Section \ref{wave decay section}, we prove \eqref{LED} by combining \eqref{unif resolv est} with an argument involving Plancherel's theorem and contour deformation. A similar strategy appears in \cite[Section 3]{vo99}.

Our methods should apply directly to some more general operators, such as the wave operator $\beta(x) \partial_t^2 - \partial_x (\alpha(x) \partial_x) + V(x)$, where $V$ is real-valued, compactly supported, and has BV. In that case, however, the residual $w_\infty$ in \eqref{LED} may be more complicated, as there may or may not be a resonance at zero, and there may also be discrete negative spectrum. See \cite[Theorem 2.9]{dz} for instance, which treats the case $V \not\equiv 0$ and $\alpha, \beta \equiv 1$.


\section{Review of BV} \label{bv review section}

To keep the notation concise, for the rest of the article, we use ``prime" notation to denote differentiation with respect to $x$, e.g., $u' \defeq \partial_x u$.

Let $f : \R \to \C$ be a function of locally bounded variation. For all $x \in \R$, put 
\begin{equation} \label{LRA}
 f^L(x) \defeq \lim_{\delta \to 0^+}f(x-\delta), \qquad  f^R(x) \defeq \lim_{\delta \to 0^+}f(x+\delta), \qquad f^A(x) \defeq (f^L(x) + f^R(x))/2,
\end{equation}
where the limits exist because both the real and imaginary parts of $f$ are a difference of two increasing functions. Recall that $f$ is differentiable Lebesgue almost everywhere, so $f (x)= f^L(x) = f^R(x) = f^A(x)$ for almost all $x \in \R$.

We may decompose $f$ as 
\begin{equation} \label{decompose f}
f = f_{r, +} - f_{r, -} + i( f_{i,+} - f_{i,-}),
\end{equation}
where the $f_{\sigma,\pm}$, $\sigma \in \{r, i\}$, are increasing functions on $\R$. Each $f^R_{\sigma,\pm}$ uniquely determines a regular Borel measure $\mu_{\sigma,\pm}$ on $\R$ satisfying $\mu_{\sigma, \pm}(x_1, x_2] = f^R_{\sigma, \pm}(x_2) -  f^R_{\sigma, \pm}(x_1)$, see \cite[Theorem 1.16]{fo}. We put
\begin{equation} \label{df}
    df \defeq \mu_{r, +} - \mu_{r, -} + i( \mu_{i,+} - \mu_{i,-}),
\end{equation}
 which is a complex measure when restricted to any bounded Borel subset. For any $a  < b$,
 
 \begin{equation} \label{ftc}
 \begin{gathered}
  \int_{(a,b]}df = f^R(b) - f^R(a),\\
  \int_{(a,b)}df = f^L(b) - f^R(a).
  \end{gathered}
 \end{equation}

We collect several properties of functions of bounded variation, which are well known, and which we use to prove Theorem \ref{nontrap thm oned} in Section \ref{wtd resolv est section}. Their proofs are deferred to the appendix.

\begin{proposition}[integration by parts] \label{ibp bv prop}
Let $f: \R \to \C$ have locally bounded variation. For any $a < b$, and any continuous $\varphi$, with $\varphi'$ piecewise continuous and $\varphi(a) = \varphi(b) = 0$, 
\begin{equation} \label{Folland IBP}
 \int_{(a,b]} \varphi df = -\int_{(a,b]} \varphi' fdx.
\end{equation}

\end{proposition}

\begin{proposition}[product rule] \label{prod rule bv prop}
Let $f, \, g : \R \to \C$ be functions of locally bounded variation. Then
\begin{equation}\label{e:prod}
 d(fg) = f^A dg + g^A df
\end{equation}
as measures on a bounded Borel subset of $\R$.
\end{proposition}

\noindent \textbf{Remark:} We note that if $f$ is continuous, then inductively applying \eqref{e:prod} yields $df^n = nf^{n-1} df$.

\begin{proposition}[chain rules] \label{chain rule bv prop}
Let $f : \R \to \R$ be continuous and have locally bounded variation. Then, as measures on a bounded Borel set of $\R$,
\begin{equation} \label{chain rule continuous}
    d(e^f) = e^{f} df.
\end{equation}
On the other hand, let $x_1, \dots x_N, r_0, r_1  \dots, r_N \in \R$, and consider the function 
\begin{equation*}
    g(x) = r_0 \mathbf{1}_{(-\infty, x_1]} + \sum_{j=1}^{N-1} r_j \mathbf{1}_{(x_j, x_{j+1}]} + r_N \mathbf{1}_{(x_N, \infty)}.
\end{equation*}
Then 
\begin{equation} \label{chain rule jumps}
   d(e^{g}) = \sum_{j=1}^N(e^{r_j} - e^{r_{j-1}}) \delta_{x_j},
\end{equation}
where $\delta_{x_j}$ denotes the dirac measure at $x_j$.
\end{proposition}
The need to treat separately the case of jump discontinuities in Proposition \ref{chain rule bv prop} was brought to the authors' attention by \cite{pi22, re22a}.

\section{Weighted resolvent estimate} \label{wtd resolv est section}
The purpose of this Section is to prove a weighted resolvent estimate for the semiclassical Schr\"odinger operator
\begin{equation} \label{defn P}
 P = P(h) \defeq -h \partial_x (\alpha(x) h \partial_x ) + V(x) - E : L^2(\R) \to L^2(\R), \qquad E, \, h > 0, 
\end{equation}
which is the key ingredient in the proof of Theorem \ref{unif resolv est thm} in Section \ref{high energy bound section}. We suppose $\alpha$ and $V$ are real-valued functions of bounded variation on $\R$, and
\begin{equation} \label{inf alpha positive}
\inf_{\R} \alpha > 0.
\end{equation}
Specifically, we show
\begin{lemma} \label{self adjointness lemma}
The operator $P : L^2(\R) \to L^2(\R)$ is self adjoint with respect to the domain 
\begin{equation} \label{D}
    \D \defeq \{u \in L^2(\R) : u, u' \in L^2(\R) \cap L^\infty(\R), \text{ and } Pu \in L^2(\R) \}, 
\end{equation}
\end{lemma}
\noindent and prove the following resolvent bound, for $h$ small, and uniformly down to $[E_{\min}, E_{\max}] \subseteq (0, \infty)$.

\begin{theorem} \label{nontrap thm oned}
Fix $[E_{\min}, E_{\max}] \subseteq (0, \infty)$ and  $\delta > 0$. Assume $\alpha, V : \R \to \R$ have bounded variation, $\alpha$ obeys \eqref{inf alpha positive}, and
\begin{equation} \label{E grtr V}
   \sup_{\R} V <  E_{\min}.
\end{equation}
Then there exist $C, h_0 > 0$, so that for all $ E \in [E_{\min}, E_{\max}],$ $h \in (0,h_0]$, and $\ep > 0$,
\begin{equation} \label{nontrap est oned}
 \|(|x| + 1)^{-\frac{1+\delta}{2}} (P(h) - i\varepsilon)^{-1} (|x| + 1)^{-\frac{1+\delta}{2}}\|_{L^2(\mathbb R) \to L^2(\mathbb R)} \le Ch^{-1}. 
\end{equation}
\end{theorem}

Since $V$ has limited regularity, we have replaced a more typical nontrapping condition, concerning the escape of trajectories $\dot{x} = 2\xi$, $\dot{\xi} = -\partial_x V$ that obey $|\xi|^2 + V(x) = E$, with the simpler condition \eqref{E grtr V}. Indeed, as $\alpha$ and $V$ have only bounded variation, the bicharacteristic flow is not necessarily well defined. Moreover, in Section \ref{high energy bound section}, we shall see that \eqref{E grtr V} is a natural assumption, given that the coefficients of the operator $H$ obey \eqref{infs positive}.


To prove Theorem \ref{nontrap thm oned}, we employ a positive commutator-style argument in the context of the spherical energy method. This strategy has long been used to prove semiclassical resolvent estimates \cite{cv, da, kv, dash20, gash22}. In fact, as we are in one dimension, we just use the pointwise energy
\begin{equation} \label{defn F}
    F(x) = F[u](x) \defeq \alpha(x)|h \partial_x u(x)|^2 +  (E- V(x))|u(x)|^2, \qquad  u \in \mathcal{D}.
\end{equation}

The goal is to construct a suitable weight function $w(x)$ so that the derivative of $wF$, in the sense of distributions, has a favorable sign. From \eqref{deriv wF with alpha} below, we see that $w$ ought to be designed so that $(w(E - V))'$ has a positive lower bound. If $V$ only has bounded variation, this derivative must be interpreted as a measure, and extra care is needed to control the point masses arising from the discontinuities of $V$ (see \eqref{lower bd first meas}).

We first give our attention to Lemma \ref{self adjointness lemma}, which is essentially well known. Our present proof is adapted from \cite[Section 2]{dash20}.

\begin{proof}[Proof of Lemma \ref{self adjointness lemma}]
Let
\begin{equation*}
  \mathcal D_{\max} \defeq \{u \in L^2(\R) : u, \, \alpha u' \text{ are locally absolutely continuous and } Pu \in L^2(\R)\},  
\end{equation*}
By, \cite[Lemma 10.3.1]{ze}, $\mathcal D_{\max}$ is dense in $L^2(\R)$. We begin by proving
\begin{equation} \label{Dmax is D}
\mathcal D_\textrm{max} = \mathcal D.
\end{equation}

Indeed, for any $a>0$ and $u \in \mathcal D_\textrm{max}$, by integration by parts and Cauchy--Schwarz,
\begin{equation*}
\begin{gathered}
\inf \alpha \int_{-a}^a|u'|^2 \le 
\int_{-a}^a \alpha u'\bar u' =  \alpha u' \overline{u} \rvert_{-a}^a  + h^{-2} \int_{-a}^a Pu \bar u - h^{-2} \int_{-a}^a Vu \bar u   \\
  \le  2 \sup \alpha \sup_{[-a,a]}|u'|\sup_{[-a,a]} |u| + h^{-2}\sup|V| \|u\|_{L^2}^2 + h^{-2}\|Pu\|_{L^2}\|u\|_{L^2} ,\\
\sup_{[-a,a]}|u|^2 =\sup_{x \in [-a,a]}  \left(|u(0)|^2 + 2 \real \int_0^x u'\bar u\right)  \le |u(0)|^2 + 2\left(\int_{-a}^a|u'|^2\right)^{1/2}\|u\|_{L^2},\\
(\inf \alpha)^2 \sup_{[-a,a]} |u'|^2 \le \sup_{[-a,a]} |\alpha u'|^2 = \sup_{x \in [-a,a]}  \left( |(\alpha u')(0)|^2 +  2 \real \int_0^x (\alpha u')'\alpha \overline{u}' \right) \\
\le |(\alpha u')(0)|^2 + 2 h^{-2} ( \sup (\alpha |V|)  \|u\|_{L^2} +  \sup \alpha  \|Pu\|_{L^2})\left(\int_{-a}^a|u'|^2\right)^{1/2}. 
\end{gathered}
\end{equation*}
This is a system of inequalities of the form $x^2 \le  A + Byz$, $y^2 \le C + Dx$, $z^2 \le E + Fx$. Thus, for any $\gamma > 0$,
\begin{equation} \label{combine ineq system}
\begin{split}
    x^2 &\le A + \frac{B}{2\gamma} + \gamma (yz)^2 \le A +   \frac{B}{2\gamma} + \gamma (C + Dx)(E+ Fx) \\
    &\le A + \frac{B}{2\gamma} + \gamma CE + \gamma \frac{ (CF)^2 + (DE)^2}{2} + (\gamma^2 + \gamma DF) x^2.
\end{split}
\end{equation}
Choosing $\gamma$ small enough allows one to absorb all the terms involving $x^2$ on the right side of \eqref{combine ineq system}, into the left side. Hence $x$, $y$ and $z$ are all bounded independently of $a$. Letting $a \to \infty$, we conclude that $u' \in L^2(\R)$ and $u, u' \in L^\infty(\R)$. Hence  $\mathcal D_{\max} \subseteq \mathcal D$. The inclusion $\mathcal{D} \subseteq \mathcal D_{\max}$ follows because $Pu \in L^2(\R)$ implies $(\alpha u')' \in L^2(\R)$, which in turns gives that $\alpha u'$ is locally absolutely continuous. 

Equip $P$  with the domain $\mathcal D_\textrm{max} = \mathcal D \subseteq L^2(\R)$. By integration by parts, $P \subseteq P^*$. But, by Sturm--Liouville theory,  $P^* \subseteq P$; see  \cite[Equation 10.3.2]{ze}. Hence $P=P^*$.

\end{proof} 

We now prove Theorem \ref{nontrap thm oned}, with the argument proceeding in two steps. First, as described above, we build a weight $w$ so that, $d(wF)$ has a desirable lower bound in the sense of measures--see \eqref{deriv wF with alpha}. This yields the Carleman estimate \eqref{penult est dim one}, which implies the resolvent estimate \eqref{final estimate one dim}.

\begin{proof}[Proof of Theorem \ref{nontrap thm oned}]
Decompose
\begin{equation*}
\begin{gathered}
dV = dV^d + dV^c, \\
d\alpha = d\alpha^d + d\alpha^c, \\
\end{gathered}
\end{equation*}
into their discrete and continuous parts. Let $J_V$, respectively $J_\alpha$ be the sets of ``positive jumps" of $V$, $\alpha$ respectively. That is $J_V$ is the set of $x$-values such that $(V^R - V^L)(x) > 0$, and similarly for $J_\alpha$. Since $V$ and $\alpha$ have bounded variation, both $J_V$ and $J_\alpha$ are at most countable. We denote by $\{x_j\}_j$ an enumeration of $J_V \cup J_\alpha$. Additionally, let
\begin{equation*}
\begin{gathered}
dV^c = dV^c_+ - dV^c_-,\\
d\alpha^c = d\alpha^c_+ - d\alpha^c_-, \\
\end{gathered}
\end{equation*}
be Jordan decompositions for $dV^c$, $d\alpha^c$ respectively.

For each $N \in \NN$, let $x_{1,N}, x_{2,N}, \dots, x_{N,N}$ be the elements of $\{x_j\}_{j=1}^N$ relabeled in increasing order. Define the function $q_{1,N}$ by 
\begin{equation} 
q_{1,N}(x) \defeq  r_{0,N} \mathbf{1}_{(-\infty, x_{1,N}]} + \sum^{N-1}_{j = 1} r_{j,N} \mathbf{1}_{(x_{j,N}, x_{j+1,N}]} + r_{N,N} \mathbf{1}_{(x_{N,N}, \infty)}, \label{q1} 
\end{equation}
where the numbers $\{r_{j,N}\}_{j=0}^N$ are defined recursively as follows:
\begin{gather}
    r_{0, N} = 0, \qquad r_{j,N} = r_{j-1,N} + \log  \max \Big\{1 + \frac{2A_{j,N}}{1 - A_{j,N}}, 1 + \frac{2B_{j,N}}{1 - B_{j,N}} \Big\}, \label{rj} \\
     A_{j,N} \defeq \frac{(V^R- V^L)(x_{j,N})}{2(E - V)^A(x_{j,N})} \in [0, 1), \qquad B_{j,N} \defeq \frac{(\alpha^R- \alpha^L)(x_{j,N}) }{2\alpha^A(x_{j,N})} \in [0, 1). \label{Aj and Bj}
\end{gather}
 When $N = 1$, we omit the summation from \eqref{q1}. Moreover, if $\{x_j\}_j$ is a finite set, we work only with a single function $q_{1,N_1}$, where $x_1 < \cdots < x_{N_1}$ is the ordering of $J_V \cup J_\alpha$. 

Since $V$ and $\alpha$ have bounded variation,
\begin{equation} \label{apply BV}
\sum_j \max\{ (V^R - V^L)(x_j), (\alpha^R - \alpha^L)(x_j)\} < \infty. 
\end{equation} 
Thus $\max q_{1,N} = r_{N,N}$ is bounded uniformly in $N$, by 
\begin{equation} \label{rNN}
    \begin{split}
    r_{N,N} &= \sum_{j =1}^N r_{j,N} - r_{j-1,N} \\
    &= \sum_{j =1}^N \log \max \Big\{1 + \frac{2A_{j,N}}{1 - A_{j,N}}, 1 + \frac{2B_{j,N}}{1 - B_{j,N}} \Big\}  \\
    &\le \sum_{j=1}^N \max \Big\{\frac{2A_{j,N}}{1 - A_{j,N}},\frac{2B_{j,N}}{1 - B_{j,N}} \Big\} \\
    &\le \sum_{j=1}^N \max \Big\{ \frac{(V^R - V^L)(x_{j,N})}{(E - V)^A(x_{j,N}) - \tfrac{1}{2}(V^R-V^L)(x_{j,N})}, \frac{(\alpha^R - \alpha^L)(x_{j,N})}{\alpha^A(x_{j,N}) - \tfrac{1}{2}(\alpha^R - \alpha^L)(x_{j,N})}  \Big \} < \infty.
    \end{split}
\end{equation}

Next, we put 
\begin{equation} \label{q2}
   q_{2}(x)  \defeq \int_{-\infty}^x \big[ k dV^c_{+}  + \tfrac{2}{\inf \alpha} d\alpha^c_{+}  + (|x'|+1)^{-1-\delta}dx' \big], 
\end{equation}
where $k > 0$ is chosen large enough so that 
\begin{equation} \label{cond k} 
 k \left(E_{\min} - \sup_{\R} V\right) \ge 1.
\end{equation}
To implement the energy method outlined in Section \ref{introduction}, we will in fact use a family of weight functions depending on $N$,
\begin{equation} \label{wN}
w(x) = w_N(x) = e^{q_{1,N}(x) + q_2(x)}, \qquad N \in \NN.
\end{equation}
According to \eqref{chain rule continuous} and \eqref{chain rule jumps},
\begin{equation} \label{dw}
    dw(x) =  \sum_{j=1}^N e^{q_2}  (e^{r_{j,N}} - e^{r_{j-1,N}}) \delta_{x_{j,N}} + w^A(\tfrac{2}{\inf \alpha} d\alpha^c_+ + k dV^c_+ +  (|x| + 1)^{-1 -\delta}).
\end{equation}

We now establish lower bounds on the measures $d(w(E-V))$ and $dw - (\alpha^A)^{-1}w^A d\alpha$, which we need in the estimate \eqref{deriv wF with alpha} below. For $d(w(E-V))$, we have, by \eqref{e:prod}, \eqref{cond k}  and \eqref{dw},
\begin{equation} \label{lower bd first meas}
\begin{split}
    d(w&(E-V)) \\
    &\ge (E-V)^Adw - w^A (dV^d + dV^c_+)  \\
    &\ge \sum_{j =1}^N e^{q_2}  \Big((E-V)^A(e^{r_{j,N}} - e^{r_{j-1,N}}) -(V^R - V^L)(\tfrac{1}{2} e^{r_{j,N}} + \tfrac{1}{2} e^{r_{j-1,N}} )  \Big)\delta_{x_{j,N}} \\
    &- \sum_{x \in J_{V} \setminus \{x_{j,N}\}_{j=1}^N} w^A (V^R - V^L) \delta_{x}\\
    &+w^A (k(E_{\min}-V)^A -1)dV^c_+ + w^A(E-V)^A(|x| + 1)^{-1 -\delta}.
\end{split}
\end{equation}
with the inequalities holding in the sense of measures. As for $dw - (\alpha^A)^{-1}w^A d\alpha$,
\begin{equation} \label{lower bd second meas}
\begin{split}
    dw& - (\alpha^A)^{-1}w^A d\alpha \\
    &\ge dw - (\alpha^A)^{-1}w^A (d\alpha^d + d\alpha^c_+)  \\
    &\ge \sum_{j =1}^N e^{q_2}  \Big((e^{r_{j,N}} - e^{r_{j-1,N}}) -\frac{(\alpha^R - \alpha^L)}{\alpha^A}(\tfrac{1}{2} e^{r_{j,N}} + \tfrac{1}{2} e^{r_{j-1.N}} )  \Big)\delta_{x_{j,N}} \\
    &- \sum_{x \in J_{\alpha} \setminus \{x_{j,N}\}_{j=1}^N} (\alpha^A)^{-1} w^A (\alpha^R - \alpha^L) \delta_{x}\\
    &+ w^A (\tfrac{2}{ \inf \alpha} - \tfrac{1}{\alpha^A} )d\alpha^c_+ + w^A(|x| + 1)^{-1 -\delta}.
\end{split}
\end{equation}
The first term in line five of \eqref{lower bd first meas} is nonnegative by \eqref{cond k}; the first term of line four of \eqref{lower bd second meas} is nonnegative since $\inf \alpha < 2\alpha^A$. Furthermore, the third line of \eqref{lower bd first meas} and the third line of \eqref{lower bd second meas}, are nonnegative by \eqref{rj} and \eqref{Aj and Bj}.

Thus we conclude 
\begin{equation} \label{final lwr bds}
\begin{gathered}
 d(w(E-V)) \ge w^A(E_{\min}-V)^A(|x| + 1)^{-1 -\delta} - \sum_{x \in J_{V} \setminus \{x_{j,N}\}_{j=1}^N} w^A (V^R - V^L) \delta_{x}, \\
 dw - (\alpha^A)^{-1}w^A d\alpha  \ge  w^A(|x| + 1)^{-1 -\delta} - \sum_{x \in J_{\alpha} \setminus \{x_{j,N}\}_{j=1}^N} (\alpha^A)^{-1} w^A (\alpha^R - \alpha^L) \delta_{x},
 \end{gathered}
\end{equation}
which are the lower bounds we shall employ in \eqref{deriv wF with alpha}.

Next, define the pointwise energy 
\begin{equation} \label{defn F with alpha}
    F(x) = F[u](x) \defeq \alpha(x)|hu'(x)|^2 +  (E- V(x))|u(x)|^2, \qquad x \in \R,
\end{equation}
with 
\begin{equation} \label{defn u}
    u = (P(h)- i\varepsilon)^{-1}(|x| + 1)^{-\frac{1+\delta}{2}} f \in \D, \qquad \ep >0,\, f \in L^2(\R).
\end{equation}
By \eqref{D}, $u, \, u' \in L^2(\R) \cap L^\infty(\R),$ and $(\alpha u')' \in L^2(\R)$. Moreover, in the calculations to follow, we work with fixed representatives of $u$ and $u'$, such that both $u$ and $\alpha u'$ are locally absolutely continuous. This is justified by \eqref{Dmax is D}.

From \eqref{e:prod}, we see that $dF$ is given by
\begin{equation*} 
    dF = h^2(\alpha u') d(\overline{u}') + h^2(\overline{u}')^A (\alpha u')' -|u|^2dV + 2(E - V)^A \real\left( u \overline{u}' \right).
\end{equation*}
Using
\begin{equation*}
    (\alpha u')' = (u')^A d\alpha + \alpha^A d(u') \implies  d(u') = \frac{(\alpha u')'}{\alpha^A} -\frac{(u')^A}{\alpha^A}   d\alpha,
\end{equation*}
we arrive at 
\begin{equation} \label{F prime with alpha}
    dF = \tfrac{h^2}{\alpha^A} (\alpha u') (\alpha \overline{u}')' + h^2(\overline{u}')^A (\alpha u')' - \tfrac{h^2}{\alpha^A}(\alpha u')(\overline{u}')^A d\alpha -|u|^2dV + 2(E - V)^A \real\left( u \overline{u}' \right).
\end{equation}

We now multiply \eqref{defn F with alpha} by $w$ and compute $d(wF)$:
\begin{equation} \label{deriv wF with alpha}
\begin{split}
    d(wF) &= F^Adw + w^AdF \\
    &= h^2 (\alpha u')(\overline{u}')^A dw +  (E - V)^A|u|^2dw \\
    &+ \tfrac{h^2}{\alpha^A} w^A (\alpha u') (\alpha \overline{u}')' + h^2 w^A (\overline{u}')^A (\alpha u')' - \tfrac{h^2}{\alpha^A}w^A(\alpha u')(\overline{u}')^A d\alpha \\
    &-w^A|u|^2dV + 2w^A(E - V)^A \real\left( u \overline{u}' \right). \\
    &= -w^A  \left(-\tfrac{h^2}{\alpha^A} (\alpha u') (\alpha \overline{u}')' - h^2 (\overline{u}')^A (\alpha u')' + 2(V - E)^A \real(u \overline{u}')- 2\real (i\varepsilon u \overline{u}') \right) \\
    &+ 2\ep w^A \imag \left(u \overline{u}'\right) + |u|^2d(w(E-V)) + h^2 (\alpha u')(\overline{u}')^A \Big(dw - w^A\tfrac{d\alpha}{\alpha^A}\Big)  \\
    &\ge -w^A  \left(-\tfrac{h^2}{\alpha^A} (\alpha u') (\alpha \overline{u}')' - h^2 (\overline{u}')^A (\alpha u')' + 2(V - E)^A \real(u \overline{u}')- 2\real (i\varepsilon u \overline{u}') \right) \\
    &+2\ep w^A \imag \left(u \overline{u}'\right)+(|x| + 1)^{-1-\delta}( (E_{\min}  - \sup_\R V)|u|^2 + h^2 (\alpha u')(\overline{u}')^A) \\
    &- \sum_{x \in J_{V} \setminus \{x_{j,N}\}_{j=1}^N} w^A |u|^2 (V^R - V^L) \delta_{x} - \sum_{x \in J_{\alpha} \setminus \{x_{j,N}\}_{j=1}^N} (\alpha^A)^{-1} h^2 w^A (\alpha u')(\overline{u}')^A (\alpha^R - \alpha^L) \delta_{x}.
  \end{split}
\end{equation}
To get lines seven and eight we plugged in \eqref{final lwr bds} and used $w^A \ge 1$.

We now integrate both sides of \eqref{deriv wF with alpha} over all of $\R$. Since $F \in L^1(\R)$ and is continuous off of a countable set, $F(x)$ tends to zero along a sequence of $x$-values tending to $+\infty$, and at which $F(x) = F^R(x) = F^L(x)$.  Similarly, $F(x) = F^R(x) = F^L(x) \to 0$ along a sequence of $x$-values tending to $-\infty$. Thus \eqref{ftc} gives $\int_\R d(wF) = 0$. Since the average values of functions that appear are equal to the functions themselves Lebesgue almost-everywhere, for each $N$, we arrive at,

\begin{equation} 
\begin{split} \label{pre penult est before limit}
   (1/&\max w)\int (|x| + 1)^{-1-\delta}\big((E_{\min}  - \sup_\R V)|u|^2 + \inf \alpha |hu'|^2\big)  \\ &\le \int  2|(P(h) - i\ep)u) \overline{u}'| + 2 \varepsilon |u u'|\\
   &+ \sum_{x \in J_{V} \setminus \{x_{j,N}\}_{j=1}^N}  |u|^2 (V^R - V^L) \delta_{x} + \sum_{x \in J_{\alpha} \setminus \{x_{j,N}\}_{j=1}^N} (\alpha^A)^{-1} h^2 (\alpha u')(\overline{u}')^A (\alpha^R - \alpha^L) \delta_{x}. 
   \end{split}
\end{equation}
Sending $N \to \infty$, recalling \eqref{apply BV} (which gives $\sup_N (\max w) < \infty$ via \eqref{rNN}), \eqref{defn u}, and\\ $u, u' \in L^\infty(\R)$, and using Young's inequality, we find  
\begin{equation} \label{pre penult est}
\begin{split}
    \int (|x| &+ 1)^{-1-\delta}\big(|u|^2 + |hu'|^2\big)\\
    & \le C \int  \frac{1}{\gamma h^2} |f|^2 + \gamma (|x| + 1)^{-1-\delta}|hu'|^2 +  2 \varepsilon |u u'| \qquad h, \, \gamma > 0. 
\end{split}
\end{equation}
Here and below, $C>0$ is a constant that may change from line to line, but it is always independent of $u$, $\ep$, and $h$. 

The second term on the right side of \eqref{pre penult est} can be absorbed into the left side by selecting $\gamma$ small enough. As for the term involving $\ep$, by Young's inequality,
\begin{equation*}
    \int |u \overline{u}'| \le \frac{1}{2 h \inf \alpha} \int |u|^2 + \frac{1}{2h} \int \alpha |h u'|^2, \qquad h > 0.
\end{equation*}
Then 
\begin{equation*}
    \begin{split}
        \int \alpha |hu'|^2 &= \real \int -h^2(\alpha u')' \overline{u} \\
        &= \real \int \left((P(h) - i\varepsilon) - V + E\right)u \overline{u} \\
        &\le\frac{1}{2} \int |(|x| + 1)^{-\frac{1+ \delta}{2}}f|^2 + \left(\frac{1}{2} + \|E_{\max} - V\|_{L^\infty}\right)\int |u|^2. 
    \end{split}
\end{equation*}
Substituting these observations and calculations into \eqref{pre penult est} gives, for $\varepsilon, h > 0$,
\begin{equation} \label{penult est dim one}
     \int (|x| + 1)^{-1-\delta}( |u|^2 + |hu'|^2)  \le \frac{C}{h^2} \int |f|^2 +  \frac{C \varepsilon}{h} \int |u|^2. 
\end{equation}

To finish, we rewrite $\varepsilon \int |u|^2$ and estimate, for any $\gamma > 0$,
\begin{equation} \label{gamma eqn dim one}
    \begin{split}
        \varepsilon \int |u|^2 &= - \imag \int (P(h) - i\varepsilon)u\overline{u} \\
        &\le \frac{1}{\gamma } \int |f|^2 + \gamma \int (|x| + 1)^{-1-\delta}|u|^2.
    \end{split}
\end{equation}
If we now take $\gamma$ sufficiently small (depending on $C$ and $h$), we may absorb the integral of \\ $(|x| + 1)^{-1-\delta}|u|^2$ in \eqref{gamma eqn dim one} into the left side of \eqref{penult est dim one}  to achieve
\begin{equation} \label{final estimate one dim}
    \int (|x| + 1)^{-1-\delta}(|u|^2 + |hu'|^2 ) \le  \frac{C}{h^2} \int |f|^2, \qquad \varepsilon >0, \, h \in (0,1].
\end{equation}
This completes the proof of \eqref{nontrap est oned}. \\
\end{proof}

\section{High frequency bound on the cutoff resolvent} \label{high energy bound section}

In this Section, we prove Theorem \ref{unif resolv est thm} as an application of Theorem \ref{nontrap thm oned}. We return to working with the operator $H : \mathcal{D}(H) \to \mathcal{H}$ as defined by \eqref{H}, where $\alpha, \beta : \R \to (0,\infty)$ are BV functions obeying \eqref{infs positive} and \eqref{perturbations of identity}. 

In that situation, $H$ is a \textit{black box Hamiltonian} in the sense of Sj\"ostrand and Zworski \cite{sz91}, as defined in \cite[Definition 4.1]{dz}. More precisely, in our setting this means the following. First, if  $u \in \mathcal{D}(H)$, then $u|_{\mathbb R\setminus [-R_0,R_0]} \in H^2(\mathbb R\setminus [-R_0,R_0])$.  Second, for any $u \in \mathcal{D}(H)$, we have $(Hu)|_{\mathbb R\setminus [-R_0,R_0]}  = - u''|_{\mathbb R\setminus [-R_0,R_0]}$. Third, any $u \in H^2(\mathbb R)$ which vanishes on a neighborhood of $[-R_0,R_0]$ is also in $\mathcal{D}(H)$. Fourth, $\textbf{1}_{[-R_0,R_0]} (H+i)^{-1}$ is compact on $\mathcal H$; this last condition follows from the fact that $\mathcal{D}(H) \subseteq H^1(\R)$.

Then, by \cite[Theorem 4.4]{dz}, for any $\chi \in C_0^\infty(\mathbb R; [0,1])$ that is identically one near $[-R_0, R_0]$, the cutoff resolvent \eqref{continued resolv} continues meromorphically $\mathcal{H} \to \mathcal{D}(H)$ from $\imag \lambda > 0$ to the complex plane. The poles of this continuation are precisely at those values $\lambda$ for which there is a solution $u$ to $Hu=\lambda^2 u$ having $u, u', Hu \in L^2_{\text{loc}}(\R)$ in the sense of distributions, and which is outgoing, i.e. obeys 
\begin{equation}\label{e:uoutgoing}
\pm x \ge R_0 \qquad \Longrightarrow \qquad u(x) = c_{\pm} e^{\pm i \lambda x},
\end{equation}
for some nonzero constants $c_{\pm}$.

Observe that $\lambda = 0$ is such a pole because we may take $u(x) = 1$ for all $x$. Observe further that this is the only pole in the closed half plane $\imag \lambda \ge 0$. Indeed, if $u$ satisfying \eqref{e:uoutgoing} solves $Hu = \lambda^2u$ with $\imag \lambda >0$, then $u \in \mathcal D(H)$ and we have $\lambda^2 \|u\|^2_{\mathcal H}  = \langle H u , u \rangle_{\mathcal H}  = \int_{\mathbb R} \alpha |u'|^2\ge 0$, which implies $\|u\|_{\mathcal H} = 0$ since $\lambda^2 \ge 0$ is impossible when $\imag \lambda > 0$. For $\lambda \in \mathbb R \setminus \{0\}$ this follows as in the proof of \cite[(2.2.12)]{dz}.


\begin{proof}[Proof of Theorem \ref{unif resolv est thm}]
Set $V_{\beta} \defeq 1- \beta$ and $ \mathcal{O} \defeq \{ \lambda \in \mathbb{C} : \real \lambda \neq 0, \text{ } \imag \lambda > 0 \}$. Note that $\supp V_\beta \subseteq [-R_0, R_0]$.  Define on $\mathcal{O}$ the following families of operators $\mathcal{H} \to \mathcal{H}$ with domain $\mathcal{D}(H)$,
\begin{gather}
\begin{split}
A(\lambda) &\defeq (\real \lambda)^{-2} \beta (H - \lambda^2)  \\
&=-(\real \lambda)^{-2} \partial_x \alpha \partial_x + V_\beta + (\imag \lambda)^2(\real \lambda)^{-2} \beta - i2 \imag \lambda(\real \lambda)^{-1}\beta - 1,  
\end{split} \label{A} \\ 
B(\lambda) \defeq -(\real \lambda)^{-2} \partial_x \alpha \partial_x + V_\beta + (\imag \lambda)^2(\real \lambda)^{-2} - 1 - i2 \imag \lambda(\real \lambda)^{-1}, \nonumber 
\end{gather}
Furthermore, define on $\mathcal{O}$ the family $\mathcal{H} \to \mathcal{H}$,
\begin{equation*}
D(\lambda) \defeq (\imag \lambda)^2(\real \lambda)^{-2}V_\beta - i2 \imag \lambda(\real \lambda)^{-1}V_\beta.
\end{equation*}

We have,
\begin{equation*}
B(\lambda) - A(\lambda) = D(\lambda).
\end{equation*}
Composing with inverses gives
\begin{equation*} 
A(\lambda)^{-1} - B(\lambda)^{-1} = B(\lambda)^{-1} D(\lambda)   A(\lambda)^{-1} \implies( I -  B(\lambda)^{-1}D(\lambda))A(\lambda)^{-1} = B(\lambda)^{-1},
\end{equation*}
Multiplying on the left and right by $\chi$ and noticing that $D(\lambda) = \chi D(\lambda) \chi$, we arrive at
\begin{equation} \label{prelim resolv id A and B}
(I - \chi B(\lambda)^{-1} \chi D(\lambda)) \chi A(\lambda)^{-1} \chi = \chi B(\lambda)^{-1} \chi, \qquad \lambda \in \mathcal{O}. 
\end{equation}

Next, choose $\lambda_0, \, \ep_0 > 0$ so that $\sup_{\R}V_\beta < 1 - \ep_0^2 \lambda_0^{-2}$. Identifying $E_{\min} \defeq 1 - \ep_0^2 \lambda_0^{-2}$, $E_{\max} = 1$, and $h \defeq|\real \lambda|^{-1}$, we see that Theorem \ref{nontrap thm oned} applies to $B(\lambda)^{-1}$. So for some $C > 0$ and a possibly larger  $\lambda_0$, we have 
\begin{equation}\label{est for B}
\| \chi B(\lambda)^{-1} \chi \|_{\mathcal{H} \to \mathcal{H}} \le C|\real \lambda|, \qquad |\real \lambda| \ge \lambda_0, \, 0 < \imag \lambda \le \ep_0 .
\end{equation}
Moreover,
\begin{equation} \label{est for D}
\|D(\lambda)\|_{\mathcal{H} \to \mathcal{H}} \le \ep_0 \|V_\beta\|_{L^\infty} \big( \frac{ 1}{\lambda^2_0} + \frac{ 2}{\lambda_0} \big), \qquad |\real \lambda| \ge \lambda_0, \, 0 < \imag \lambda \le \ep_0.
\end{equation}
Thus, increasing $\lambda_0$ again if needed, we can invert $(I - \chi B(\lambda)^{-1} \chi D(\lambda))$ by a Neumann series when $|\real \lambda| \ge \lambda_0$, $0 < \imag \lambda < \ep_0$. From \eqref{prelim resolv id A and B}, \eqref{est for B}, and \eqref{est for D}, we find
\begin{equation} \label{Neumann A and B}
\chi A(\lambda)^{-1} \chi = \left(\sum_{k=0}^\infty (\chi B(\lambda)^{-1} \chi D(\lambda))^k \right) \chi B(\lambda)^{-1} \chi, \quad  |\real \lambda | \ge \lambda_0, \, 0 < \imag \lambda \le  \ep_0.  
\end{equation}

Since
\begin{equation*}
\chi R(\lambda) \chi = (\real \lambda)^{-2}  \chi A(\lambda)^{-1} \chi  \beta, \qquad \lambda \in \mathcal{O},
\end{equation*}
\eqref{unif resolv est} follows from \eqref{est for B}, \eqref{est for D}, and \eqref{Neumann A and B}, at least when $|\real \lambda | \ge \lambda_0, \, 0 \le \imag \lambda \le  \ep_0$. To get \eqref{unif resolv est} for $|\real \lambda | \ge \lambda_0, \, |\imag \lambda| \le  \ep_0$, we appeal to a resolvent identity argument due to Vodev \cite[Theorem 1.5]{vo14}, which was adapted to the non-semiclassical (see, for instance, \cite[Lemma 5.1]{sh18}). It yields, for possibly smaller $\ep_0$, holomorphicity of $\chi R(\lambda) \chi$ in $|\real \lambda | \ge \lambda_0, \, - \ep_0 \le \imag \lambda  \le  0$, along with a bound of the form \eqref{unif resolv est} there.  
\\
\end{proof}

To conclude this section, we consider the two by two matrix operator
\begin{equation*}
    G \defeq -i \begin{pmatrix} 0 & 1 \\ -H & 0 \end{pmatrix} : \mathcal{D}(H) \oplus \mathcal{H} \to \mathcal{H} \oplus \mathcal{H},
\end{equation*}
which arises naturally from rewriting \eqref{wave eqn} as a first order system. A short computation yields, 
\begin{equation} \label{inv G plus lambda}
(G + \lambda)^{-1} = \begin{pmatrix} -\lambda R(\lambda) & -iR(\lambda) \\ i \lambda^2 R(\lambda) + i &  -\lambda R(\lambda)  \end{pmatrix}, \qquad \imag \lambda > 0.
\end{equation}

The following Corollary of Theorem \ref{unif resolv est thm} is essentially well-known, and is an important input to the proof of Theorem \ref{LED thm} in Section \ref{wave decay section}. We give the proof by recalling several results from \cite{bu03, vo14, dz}.
\begin{corollary} \label{matrix op cor}
Let $\chi \in C^\infty_0(\R; [0,1])$ be identically one near $[-R_0, R_0]$. The operator 
\begin{equation} \label{matrix op cutoffs}
    \chi (G + \lambda)^{-1} \chi \defeq \begin{pmatrix} -\lambda \chi R(\lambda) \chi & -i\chi R(\lambda)\chi \\ i \lambda^2 \chi R(\lambda) \chi + i \chi^2 &  -\lambda \chi R(\lambda) \chi  \end{pmatrix} : H^1(\R) \oplus L^2(\R) \to H^1(\R) \oplus L^2(\R)
\end{equation}
continues meromorphically from $\imag \lambda > 0$ to $\C$. It has no poles on $\mathbb R \setminus \{0\}$ and at $\lambda = 0$ it has a simple pole: more precisely, if $w_0 \in H^1(\mathbb R)$ and $w_1 \in L^2(\mathbb R)$, then 
\begin{equation}\label{e:residuecomp}
\lim_{\lambda \to 0} \lambda \chi (G+\lambda)^{-1} \chi \begin{pmatrix} w_0 \\ w_1  \end{pmatrix} =  \begin{pmatrix}  -i \lim_{\lambda \to 0} \lambda \chi  R(\lambda) \chi w_1\\ 0   \end{pmatrix} = \begin{pmatrix} \frac 12\langle \chi, w_1 \rangle_{\mathcal H} \chi \\ 0  \end{pmatrix}.
\end{equation}

Furthermore, there exist $C,\, \lambda_0, \, \ep_0 > 0$ so that
 \begin{equation} \label{matrix op unif bd}
       \|\chi (G + \lambda)^{-1} \chi \|_{H^1(\R) \oplus L^2(\R) \to H^1(\R) \oplus L^2(\R)} \le C, 
 \end{equation}
 whenever $|\real \lambda| \ge \lambda_0$, and $|\imag \lambda| \le \ep_0$.
\end{corollary}

\begin{proof}
As described above, by \cite[Theorem 4.4]{dz} and the proof of \cite[(2.2.12)]{dz}, the operator $\chi R(\lambda) \chi : L^2(\R) \to \mathcal{D}(H)$ continues meromorphically from $\imag \lambda >0$ to $\mathbb{C}$, and has no poles in $\R \setminus \{0\}$. This implies that each entry of \eqref{matrix op cutoffs} continues meromorphically as an operator between the appropriate spaces, again without poles in $\R \setminus \{0\}$.

Next, as in the proof of \cite[Theorem 2.7]{dz}, \eqref{perturbations of identity} implies that near $\lambda  = 0$,
\begin{equation} \label{resolv near zero}
\chi R(\lambda) \chi w_1 =  \frac i {2 \lambda} \langle \chi, w_1 \rangle_{\mathcal H} \chi + A(\lambda)w_1, 
\end{equation}
where $A(\lambda) : \mathcal{H} \to \mathcal{D}(H)$ is holomorphic near zero, and hence we have \eqref{e:residuecomp}.

With \eqref{unif resolv est} already in hand, to establish \eqref{matrix op unif bd}, it suffices to supply $\lambda_0, \ep_0 > 0$ so that 
\begin{gather}
\lambda^2 \chi R(\lambda) \chi +  \chi^2 = \chi H R(\lambda) \chi:  H^1(\R) \to L^2(\R) , \label{two one entry} \\
\lambda \chi R(\lambda) \chi : H^1(\R) \to H^1(\R), \label{one one entry}
\end{gather}
are uniformly bounded for $ |\real \lambda| \ge \lambda_0$ and $|\imag \lambda| \le \ep_0$. When $ |\real \lambda| \ge \lambda_0$ and $0 < \imag \lambda \le \ep_0$ this follows from the proof of \cite[Proposition 2.4]{bu03}, see in particular \cite[(2.14), (2.17), and (2.19)]{bu03}. To extend these bounds to strips below the real axis, we use once more Vodev's resolvent identity (\cite[Theorem 1.5]{vo14} and \cite[Lemma 5.1]{sh18}).

\end{proof}

\section{Wave decay} \label{wave decay section}

\begin{proof}[Proof of Theorem \ref{LED thm}]

This section follows part of Section 3 of \cite{vo99}. 

Recall that we use $w(t)$ to denote the solution \eqref{soln spectral thm} to \eqref{wave eqn}, with initial data $w_0 \in \mathcal{D}(H)$ and $w_1 \in \mathcal{D}(H^{1/2})$. We have $\supp w_0, \, \supp w_1 \subseteq (-R,R)$, and the coefficients of \eqref{wave eqn} obey \eqref{infs positive} and \eqref{perturbations of identity}. We want to show that the local energy $\| w(\cdot, t) - w_\infty \|_{H^1(-R_1,R_1)}  + \|\partial_t w(\cdot, t)  \|_{L^2(-R_1,R_1)}$ decays exponentially, for a suitable constant $w_\infty$.

Choose $\chi \in C_0^\infty(\mathbb R; [0,1])$ such that $\chi =1$ near $[-R_1, R_1] \cup [-R, R] \cup [R_0, R_0]$ ($R_0$ given as in \eqref{perturbations of identity}). Recall from Corollary \ref{matrix op cor}  that there exist $C$, $\lambda_0$, $\varepsilon_0 > 0$ such that
\begin{equation*}
\|\chi(G + \lambda)^{-1}\chi f\| \le C\|f\|,
\end{equation*}
whenever $|\real \lambda | \ge \lambda_0$ and $|\imag \lambda| \le \varepsilon_0$, where here and for the rest of this section all norms are $H^1(\R) \oplus L^2(\R)$ unless otherwise specified.

We have
\[\begin{split}
w(t) &= \cos(t H^{1/2}) w_0 + \sin(tH^{1/2})H^{-1/2} w_1,\\
\partial_t w(t) &= -\sin(tH^{1/2}) H^{1/2} w_0 + \cos(tH^{1/2})w_1, \\
\partial_t^2 w(t) &= - H w(t).
\end{split}\]
Consequently, after defining
\[
 f \defeq \left( \begin{array}{c}w_0\\w_1 
\end{array} \right), \qquad U(t) f \defeq \left( \begin{array}{c} 
w(t) \\
\partial_t w(t)\end{array}\right),
\]
we have
\begin{equation}\label{e:utfbound}
\|U(t) f \| \le C\|f\|, \qquad \partial_t  U(t) f =i G U(t) f, \qquad U(t)U(s) f =  U(t+s)f,
\end{equation}
for all real $t$ and $s$, for some $C>0$ independent of $t$ and $f$. (Note that $U(t)f$ is still defined even if only $w_0 \in \mathcal{D}(H^{1/2})$, $w_1 \in \mathcal{H}$.) 

Take $\varphi \in C^\infty(\mathbb R; [0,1])$ which is $0$ on $(-\infty,1]$ and $1$ on $[2,\infty)$ and put
\[
 W(t) f \defeq \varphi(t) U(t) f =   \int_{\imag \lambda = \varepsilon}  e^{-it\lambda} \check W(\lambda) \, d \lambda, \qquad \check W(\lambda) \defeq \frac 1 {2\pi} \int_{\R} e^{is\lambda}W(s)fds.
\]
Since $\partial_t W(t)f = \varphi'(t) U(t)f + i G W(t)f$ we get
\[
 W(t) f = \int_{\imag \lambda = \varepsilon}  e^{-it\lambda} (G + \lambda)^{-1}(i\varphi' U f)\check{~}(\lambda)\,d \lambda.
\]

Since $\supp w_0, \, \supp w_1 \subseteq (-R,R)$, by finite speed of propagation for the wave equation, and increasing $R > 0$ if necessary, we have that, $x \mapsto U(t)f$ is supported in $(-R,R)$ for all $t \in [0,2]$. By continuity of integration, the same is true of $x \mapsto (i\varphi' U f)\check{~}(\lambda)$ for every $\lambda$. Hence $ \lambda \mapsto (i\varphi' U f)\check{~}(\lambda)$ is entire and rapidly decaying as  $|\real \lambda| \to \infty$ with $|\imag \lambda|$ remaining bounded and further $(i\varphi' U f)\check{~}(\lambda) = \chi(i\varphi' U f)\check{~}(\lambda)$ . Take $\varepsilon \in(0,\varepsilon_0)$ small enough that $\lambda = 0$ is the only pole of $\chi R(\lambda)\chi $ (and hence also of $\chi (G+\lambda)^{-1}\chi$ by \eqref{inv G plus lambda}) in the half plane $\imag \lambda \ge -\varepsilon$. By deformation of contour,
\[
\chi  W(t)f =  \lim_{\lambda \to 0} \lambda \chi (G+\lambda)^{-1} \chi \int_{\R} \varphi'(s) U(s) f\,ds + \int_{\imag \lambda = -\varepsilon}  e^{-it\lambda} \chi (G + \lambda)^{-1}\chi (i\varphi' U f)\check{~}(\lambda) d \lambda. 
\]
To simplify this, use \eqref{e:residuecomp} and put
\[
 W_1(t)f := \int_{-\infty}^\infty e^{-it\lambda} (G + \lambda - i \varepsilon)^{-1} (i\varphi' U f)\check{~}(\lambda - i \varepsilon)\, d \lambda,
\]
to obtain
\[
\chi W(t) f = \begin{pmatrix} \frac 12 \chi \int_{\R} \int_0^2 \beta(x) \chi(x) \varphi'(s) \partial_sw(s,x)ds dx \\ 0 \end{pmatrix} + e^{-\varepsilon t} \chi W_1(t)f. 
\]
To simplify the first term, we integrate by parts in $s$, using $\varphi' = -(1-\varphi)'$, to obtain
\[
\int_{\R} \int_0^2  \beta(x) \chi(x) \varphi'(s) \partial_sw(s,x)\,ds\,dx =\int_{\R}\beta \chi w_1 +\int_{\R} \int_0^2   \beta(x) \chi(x) (1-\varphi(s)) \partial_s^2w(s,x)\,ds\,dx. 
\]
Now observe that $\partial_s^2 w = - H w$ and $\langle \chi, H w(s) \rangle_{\mathcal{H}} =0$ for $s \in [0,2]$ (the latter fact following from $\chi = 1$ near $[-R, R]$ and $\supp w(s) \subseteq (-R, R)$ for $s \in [0,2]$). Thus
\[
\chi W(t) f = \frac 12  \begin{pmatrix} \langle \chi, w_1 \rangle_{\mathcal{H}} \chi  \\ 0  \end{pmatrix}+ e^{-\varepsilon t} \chi W_1(t)f.
\]

It now suffices to show that
\[
\|\chi W_1(t)f\| \le C e^{\varepsilon t/2}\|f\|.
\]
To prove this, we first use Plancherel's theorem, along with the fact that by \eqref{e:utfbound}, the operator norm  $\|U(t)\|_{H^1(\R) \oplus L^2(\R) \to H^1(\R) \oplus L^2(\R)}$ is uniformly bounded for all $t \in \mathbb R$, as well as the fact that by Corollary \ref{matrix op cor}, for any $\varepsilon>0$ small enough, the operator norm $\|\chi (G+\lambda-i\varepsilon)^{-1} \chi\|_{H^1(\R) \oplus L^2(\R) \to H^1(\R) \oplus L^2(\R)}$ is uniformly bounded for all $\lambda \in \mathbb R$, to obtain
\begin{equation}\label{e:planch}\begin{split}
\int  \|\chi W_1(t)f\|^2 \, dt &= C \int \|\chi (G + \lambda - i \varepsilon)^{-1} (\varphi' U f)\check{~}(\lambda - i \varepsilon)\|^2 \, d \lambda \\
&\le C \int \|(\varphi' U f)\check{~}(\lambda - i \varepsilon)\|^2 \, d \lambda \\ &= C \int e^{2\varepsilon t}  \|\varphi'(t) U(t) f\|^2 \, d t \le C \|f\|^2.
\end{split}\end{equation}
Next, compute
\[
 (\partial_t - i G) \chi W_1(t)f = - i [G,\chi] W_1(t)f + \varepsilon \chi W_1(t)f - i \chi \int e^{-it\lambda }(i\varphi' U f)\check{~}(\lambda - i \varepsilon)\, d \lambda \qefed \widetilde W_1(t)f.
\]
Integrating both sides of  $\partial_s( U(t-s) \chi W_1(s)f) = U(t-s)\widetilde W_1(s)f$ from $s=0$ to $s=t$ gives
\[
 \chi W_1(t)f  = U(t)\chi W_1(0)f  + U(t)\int_0^t U(-s)\widetilde W_1(s)f\,ds.
\]
Thus
\[
 \|\chi W_1(t) f\| \le C \big(\|f\| +  \int_0^t \|\widetilde W_1(s)f\| ds \big) \le C\big( \|f\| +  t^{1/2} \Big(\int_0^t \|\widetilde W_1(s)f\|^2 ds\Big)^{1/2} \big).
\]
Now check that, since $\|[G,\chi] W_1(t)f\| \le C\|W_1(t)f\|$, calculating as in \eqref{e:planch}, we obtain \\ $\int \|\widetilde W_1(s)f\|^2\,ds \le C \|f\|^2$, and hence
\[
\|\chi W_1(t)f\| \le C (1 + t^{1/2}) \|f\|
\]
as desired.\\
\end{proof}

\section{Acknowledgments}

K.D. gratefully acknowledges support under NSF grant DMS 1708511. J.S. gratefully acknowledges support from four sources: ARC grant DP180100589, NSF grant DMS 1440140 while in residence at the Mathematical Sciences Research institute in Berkeley, CA, NSF grant DMS 2204322, and a University of Dayton Catholic Intellectual Tradition grant. Thanks also to the anonymous referee for helpful comments and corrections.

\appendix

\section{Characterization of $\mathcal{D}(H^{1/2})$} \label{H1 appendix}

In this Appendix we show

\begin{lemma}[{\cite{re22b}}]
\label{square root lemma}
It holds that $\mathcal{D}(H^{1/2}) = H^1(\R)$, and that\\ $u \mapsto \|u\|_{H^1}$, $u \mapsto (\|u\|^2_{\mathcal{H}} + \|H^{1/2} u\|^2_{\mathcal{H}})^{1/2}$ are equivalent norms.
\end{lemma}

\begin{proof}
First, recall the well-known fact that $\mathcal{D}(H^{1/2})$ equals the form domain associated to $H$, namely, $\mathcal{D}(H^{1/2})$ is the completion of $\mathcal{D}(H)$ with respect to the norm $\| u\|^2_{+1} \defeq \langle Hu,u \rangle_{L^2} + \langle u,u \rangle_{L^2}$. On $\mathcal{D}(H)$, it's clear that there exist $C, c> 0$ so that $ c\| \cdot \|^2_{H^1} \le \| \cdot \|^2_{+1} \le C\| \cdot \|^2_{H^1}$.

If $u \in \mathcal{D}(H^{1/2})$, then there exists a $\|\cdot\|_{+1}$-Cauchy sequence $u_j \in \mathcal{D}(H) $ converging to $u$ in $\mathcal{H}$ (or, equivalently, converging to $u \in L^2(\R)$). Because $\|\cdot\|_{+1}$ and $\| \cdot \|_{H^1}$ are equivalent on $\mathcal{D}(H)$, we get that the $u_j$ are also $\|\cdot\|_{H^1(\R)}$-Cauchy. By completeness of $H^1(\R)$, we conclude $u \in H^1(\R)$. We also have 
\begin{equation*}
\|H^{1/2}u \|^2_{\mathcal{H}} = \lim_{j \to \infty}\|H^{1/2}u_j \|^2_{\mathcal{H}} = \lim_{j \to \infty} \langle Hu_j, u_j \rangle_{\mathcal{H}} \le C \lim_{j \to \infty} \|u_j \|^2_{H^1} = C \| u\|^2_{H^1},
\end{equation*}
where the first equals sign follows since $H^{1/2}$ is a closed operator.

To show  $H^1(\R) \subseteq \mathcal{D}(H^{1/2})$, first suppose  $u \in H^1(\R)$ has compact support. Approximate $\alpha u'$ in $L^2(\R)$ by $\tilde{v}_j \in C^\infty_0(\R)$ which have support in a fixed compact set. Choose $\varphi_0 \in C^\infty_0(\R)$ with $\int \varphi_0/ \alpha = 1$, and put
\begin{equation*}
    v_j \defeq \tilde{v}_j - \big(\int \tilde{v}_j/ \alpha \big) \varphi_0.
\end{equation*}
Then $\int v_j/\alpha = 0$ and the $v_j/\alpha \to u'$ in $L^2(\R)$ since $\int u' = 0$. We clearly have \\$u_j \defeq \int_{-\infty}^x v_j/\alpha \in \mathcal{D}(H)$. Moreover, because $\int_{-\infty}^x v_j/\alpha \to \int_{-\infty}^x u' = u(x)$ locally uniformly in $x$, it follows that $u_j \to u$ in $H^1(\R)$, and that the $u_j$ are $\| \cdot \|_{+1}$-Cauchy. Hence $u \in \mathcal{D}(H^{1/2})$.

For general $u \in H^1(\R)$, choose a sequence $\tilde{u}_j$ of compactly supported functions with \\ $\|\tilde{u}_j - u\|_{H^1} \le 2^{-j-1}$. For each $j$, use the construction of the previous paragraph to find $u_j \in \mathcal{D}(H)$ with $\|u_j - \tilde{u}_j\|_{H^1} \le 2^{-j-1}$. Then the $u_j \to u$ in $H^1(\R)$ and
\begin{equation*}
   \| u_{j} - u_{k} \|^2_{+1} \le C\|u_j - u_k\|_{H^1} \to 0 \qquad \text{as $j,\, k \to \infty$.} 
\end{equation*}
Thus $u \in \mathcal{D}(H^{1/2})$ and
\begin{equation*}
    c\|u\|^2_{H^1} = c\lim_{j \to \infty}\|u_j\|^2_{H^1} \le  \lim_{j \to \infty} \| u_j\|^2_{+1} = \lim_{j \to \infty} \big( \| H^{1/2} u_j\|^2_{\mathcal{H}} + \|  u_j\|^2_{\mathcal{H}} \big) = \| H^{1/2} u\|^2_{\mathcal{H}} + \|  u\|^2_{\mathcal{H}}.
\end{equation*}
\end{proof}

\section{Elementary properties of BV functions} \label{BV appendix}

This appendix collects some facts about functions of bounded variation which can be found in \cite{vh} and \cite{afp}. The main results are the integration by parts formula \eqref{Folland IBP}, the product rule \eqref{e:prod}, and the chain rules \eqref{chain rule continuous} and \eqref{chain rule jumps}. The books \cite{vh} and \cite{afp} are mostly concerned with higher dimensional problems, so we present proofs for the much simpler one dimensional case here.

We continue to use the notation \eqref{LRA} and \eqref{df} from Section \ref{bv review section}. For $\psi  \in L^1(\mathbb R)$ compactly supported and satisfying $\int \psi = 1$, and for $\varepsilon >0 $, let 
\begin{equation}\label{e:reg}
 f_\varepsilon(x) =  \int f(x-\varepsilon y)\psi(y)dy = \varepsilon^{-1} \int f(y) \psi(\varepsilon^{-1}(x-y))dy. 
\end{equation}
Then, accordingly as $\psi$ is supported in $[0,\infty)$ or supported in $(-\infty,0]$ or even, we have
\begin{equation}\label{e:limconv}
 \lim_{\varepsilon \to 0^+} f_\varepsilon = f^L \  \textrm{or } f^R \ \textrm{or } f^A, \qquad \text{pointwise on $\R$.}
\end{equation}
Indeed, use the dominated convergence theorem in the first two cases and average them to get the third case.

\begin{proof}[Proof of Proposition \ref{ibp bv prop}]
The integration by parts formula \eqref{Folland IBP} follows as in the proof of \cite[Theorem 3.36]{fo}. Indeed, let $\Omega = \{(x,y) \in \R^2: a < x \le y \le b \}$. Since $\varphi$ is continuous and $\varphi'$ is piecewise continuous, it holds that $d\varphi = \varphi' dx$. Using Fubini's theorem, we evaluate the product measure $df \times d\varphi$ two different ways,
\begin{equation*}
    \begin{gathered}
    \int_{(a,b] \times (a,b]} \mathbf{1}_\Omega (x,y) df(x) \times d\varphi(y)
    = \int_{(a,b]} \int_{(a,y]} df(x) d\varphi(y) \\
    = \int_{(a,b]} (f^R(y) - f^R(a)) \varphi'(y)dy = \int_{(a,b]} f(y) \varphi'(y) dy,
    \end{gathered}
\end{equation*}
where we used that $f^R = f$ Lebesgue almost everyone, and that the boundary terms vanish since $\varphi(a) = \varphi(b) = 0.$ Similarly,
\begin{equation*}
\begin{gathered}
    \int_{(a,b] \times (a,b]} \mathbf{1}_\Omega (x,y) d_f(x) \times d\varphi(y)
    = \int_{(a,b]} \int_{[x,b]}  d\varphi(y) df(x)
    = -\int_{(a,b]}  \varphi(x)df(x).
    \end{gathered}
    \end{equation*}
\end{proof}

\begin{proof}[Proof of Proposition \ref{prod rule bv prop}]
Let $\psi \in C_0^\infty(\mathbb R)$ be an even function satisfying $\int \psi =1$. For any $\varepsilon>0$, define $f_\varepsilon$ by \eqref{e:reg}, and for any, $\eta>0$ define $g_\eta$ similarly. Then 
\begin{equation}\label{e:prodreg}
 (f_\varepsilon g_\eta)' = f_\varepsilon (g_\eta)' + g_\eta (f_\varepsilon)'.
\end{equation}
We now show that taking $\eta \to 0^+$ and then $\varepsilon \to 0^+$ in \eqref{e:prodreg} gives \eqref{e:prod}. 
Let $\varphi \in C_0^\infty(\mathbb R)$.  First, by integration by parts,
\[
 \lim_{\varepsilon \to 0^+}  \lim_{\eta \to 0^+} \int \varphi (f_\varepsilon g_\eta)'dx = -  \lim_{\varepsilon \to 0^+}  \lim_{\eta \to 0^+} \int \varphi' f_\varepsilon g_\eta dx.
 \]
Then, we observe that $\int \varphi' f_\varepsilon g_\eta dx \to \int \varphi' f_\varepsilon g dx$ by the dominated convergence theorem. Indeed, \\ $g_\eta \to g^A  \stackrel{a.e.}{=} g $ by \eqref{e:limconv}, and $|\varphi' f_\varepsilon g_\eta|$ is uniformly bounded for $\varepsilon$ fixed and $\eta$ small. Similarly, \\ $\int \varphi' f_\varepsilon g dx \to \int \varphi' f g dx.$ Finally, $-\int \varphi' f g dx = \int \varphi d(f g)$ by \eqref{Folland IBP}.

Next
\begin{equation*}
\begin{split}
 \lim_{\varepsilon \to 0^+}  \lim_{\eta \to 0^+}  \int \varphi f_\varepsilon g_\eta' dx &= -\lim_{\varepsilon \to 0^+}  \lim_{\eta \to 0^+}  \int (\varphi f_\varepsilon)' g_\eta dx \\
 & =  -\lim_{\varepsilon \to 0^+}   \int (\varphi f_\varepsilon)' g dx \\
 &=  \lim_{\varepsilon \to 0^+}   \int \varphi f_\varepsilon dg  \\
 &= \int  \varphi f^A dg.
\end{split}
\end{equation*}
For the first equal sign, we integrate by parts; for the second, we use the dominated convergence theorem, as in the previous paragraph. The third equal sign follows from \eqref{Folland IBP}, and the fourth from another application of the dominated convergence theorem (and \eqref{e:limconv}).

Continuing, by \eqref{e:reg}, \eqref{Folland IBP} and Fubini's theorem,
\begin{equation} \label{exchange ep}
\begin{split}
 \int \varphi g^A(f_\varepsilon)'dx &=  \int \varphi(x) g^A(x)  \varepsilon^{-2} \left[\int \psi'(\varepsilon^{-1}(x-y))f(y) dy \right]dx \\
&=   \int \varphi(x) g^A(x)  \varepsilon^{-1} \left[\int \psi(\varepsilon^{-1}(x-y))df(y) \right]dx\\
 &=  \varepsilon^{-1}\int\left[\int \varphi(x) g^A(x) \psi(\varepsilon^{-1}(y-x))dx \right]df(y) =   \int (\varphi g^A)_\varepsilon df,
\end{split}
\end{equation}
where for the third equal sign we used that $\psi$ is even.
Since $\varphi$ and $\psi$ have compact support, the integrals against $df$ make sense, and the application of Fubini's theorem is justified (even though $df$ may be finite only after it is restricted to a bounded Borel set). Finally,
\[
  \lim_{\varepsilon \to 0^+}  \lim_{\eta \to 0^+}  \int \varphi g_\eta (f_\varepsilon)'dx = \lim_{\varepsilon \to 0^+} \int \varphi g^A(f_\varepsilon)'dx =\lim_{\varepsilon \to 0^+} \int (\varphi g^A)_\varepsilon df= \int \varphi g^A df,
\]
by the dominated convergence theorem, \eqref{e:limconv}, and \eqref{exchange ep}.\\
\end{proof}

\begin{proof}[Proof of Proposition \ref{chain rule bv prop}]
Using the decomposition \eqref{decompose f}, we see that  $e^f = e^{f_{r,+}} e^{-f_{r,-}}$ has locally bounded variation, as it is a product of functions of locally bounded variation.

Let $\varphi, \psi \in C_0^\infty(\mathbb R)$, with $\psi$ even and satisfying $\int \psi =1$. To show \eqref{chain rule continuous}:

\begin{equation*}
\begin{split}
    \int \varphi d(e^f) &= - \int e^{f} \varphi' dx\\
    &= -\int \lim_{N \to \infty}  \sum_{n = 0}^N \frac{f^n}{n!} \varphi' dx \\
    &= - \lim_{N \to \infty} \sum_{n = 0}^N \int   \frac{(f^n)}{n!} \varphi' dx \\
    &=  \lim_{N \to \infty} \Big( \sum_{n=1}^N \int \frac{\varphi}{n!} df^n - \sum_{n=0}^N \int d (\varphi \frac{f^n}{n!}) \Big) \\
    &=  \lim_{N \to \infty} \sum_{n=1}^N \int \frac{\varphi}{(n-1)!} f^{n-1} df \\
    &= \int \varphi e^{f} df.\\
\end{split} 
\end{equation*}

The first equal sign follows from \eqref{Folland IBP}. The third and sixth equal signs use the dominated convergence theorem; the fourth follows by \eqref{e:prod}, and the fifth by \eqref{ftc} and the Remark after \eqref{e:prod}.

For \eqref{chain rule jumps}, we first note that, because $g$ has locally bounded variation, so does $e^g$. We compute, 
\begin{equation*}
\begin{split}
    \int \varphi d(e^g) &= - \int e^{g} \varphi' dx\\
    &= - \int_{-\infty}^{x_1} e^{r_0} \varphi' dx -  \sum_{j=1}^{N-1} \int_{x_{j}}^{x_{j+1}} e^{r_j}  \varphi' dx - \int_{x_N}^\infty e^{r_N} \varphi' dx \\
    &=\sum_{j=1}^N (e^{r_j} - e^{r_{j-1}})\varphi(x_j).
\end{split} 
\end{equation*}
\\
\end{proof}

\end{document}